\documentclass[12pt]{amsart}
\usepackage[margin=1in]{geometry}
\usepackage{amsfonts}
\usepackage{amssymb}
\usepackage[dvips]{graphics}
\usepackage{epsfig}
\usepackage{euscript}
\usepackage{color}
\usepackage{tikz}
\usepackage{mathrsfs}
\usepackage [all,arc,curve,color, arrow, matrix,frame]{xy}

\usepackage[colorlinks ,pdfstartview=FitB]{hyperref}
\newtheorem{thm}{Theorem}[section]
\newtheorem{cor}[thm]{Corollary}
\newtheorem{lemma}[thm]{Lemma}
\newtheorem{prop}[thm]{Proposition}
\theoremstyle{definition}
\newtheorem{defn}[thm]{Definition}
\theoremstyle{remark}
\newtheorem{remark}[thm]{Remark}

\numberwithin{equation}{section}
\numberwithin{figure}{section}
\def\Ran{{\mathrm{Ran\,}}}

\newcommand{\cH}{{\mathcal H}}
\newcommand{\cD}{{\mathcal D}}

\def\Span{{\mathrm{span\,}}}

\newcommand{\be}{\begin{equation}}
	\newcommand{\ee}{\end{equation}}
\newcommand{\bea}{\begin{eqnarray}}
	\newcommand{\eea}{\end{eqnarray}}
\newcommand{\clos}{\mbox{\rm clos}}
\newcommand{\ben}{\begin{eqnarray}}
	\newcommand{\een}{\end{eqnarray}}

\newcommand{\C}{\mathbb{C}}
\newcommand{\R}{\mathbb{R}}

\newcommand{\N}{\mathbb{N}}

\newcommand{\eps}{\varepsilon}
\renewcommand{\Re}{\operatorname{Re}}
\renewcommand{\Im}{\operatorname{Im}}
\newcommand{\llangle}{\left\langle}
\newcommand{\rrangle}{\right\rangle}

\DeclareMathOperator{\Imag}{Im}
\DeclareMathOperator{\Real}{Re}

\DeclareMathOperator{\spann}{span}

\def\idty{{\mathchoice {\mathrm{1\mskip-4mu l}} {\mathrm{1\mskip-4mu l}} %
		{\mathrm{1\mskip-4.5mu l}} {\mathrm{1\mskip-5mu l}}}}

\numberwithin{equation}{section}

\title[Complete non-selfadjointness]{Complete non-selfadjointness for Schr\"{o}dinger operators on the semi-axis}

\author[Fischbacher]{Christoph Fischbacher}
\address{Department of Mathematics,	Baylor University, Sid Richardson Bldg., 1410 S.\ 4th Street, Waco, TX 76706, USA}
\email{C\_Fischbacher@baylor.edu}

\author[Naboko]{Serguei Naboko}
\address{Department of Math.~Physics, Institute of Physics, St.~Petersburg State University, 1 Ulianovskaia, St.~Petergoff, St.~Petersburg, 198504, Russia}

\author[Wood]{Ian Wood}
\address{School of Mathematics, Statistics and Actuarial Sciences,
	University of Kent, Canterbury, CT2 7FS, UK}
\email{i.wood@kent.ac.uk}

\begin{document}
	
	\begin{abstract}
		In this note we investigate complete non-selfadjointness for all maximally dissipative extensions of a Schr\"odinger operator on a half-line with dissipative bounded potential and dissipative boundary condition. We show that all maximally dissipative extensions that preserve the differential expression are completely non-selfadjoint. However, it is possible for maximally dissipative extensions to have a one-dimensional reducing subspace on which the operator is selfadjoint. We give a characterisation of these extensions and the corresponding subspaces and present a specific example. 
	\end{abstract}
	\maketitle
	\begin{center}
		\today
	\end{center}
	\section{Introduction}
	
	We investigate complete non-selfadjointness for all maximally dissipative extensions of a Schr\"odinger operator on a half-line with dissipative bounded potential and dissipative boundary condition. An operator is completely non-selfadjoint if it has no non-trivial reducing subspace on which it acts as a selfadjoint operator; see Section \ref{sec:bg} for precise definitions.
	Complete non-selfadjointness is an important property of an operator on a Hilbert space which, in particular, plays a crucial role in the construction of a selfadjoint dilation of a maximally dissipative operator \cite{BMNW20,SFBK10}. 
	It is often a surprisingly difficult property to prove. 
	
	For limit-point Schr\"{o}dinger operators on a half-line with dissipative boundary condition, but real potential, constructions of the selfadjoint dilation and results on complete non-selfadjointness  can be found in Pavlov's work  \cite{Pav75, Pav76}. These results are extended to the limit circle case in \cite{All91}. Divergence-form operators on a bounded interval, again with real potentials, were treated in \cite{KNR03}. In 3-dimensional space, complete non-selfadjointness can be obtained for Schr\"{o}dinger operators with dissipative potentials using Holmgren's theorem, see \cite{Pav77}. Only fairly recently the result has been obtained for the one-velocity transport operator \cite{RT11}. Though the result will not be surprising to experts in the field, to the best of our knowledge there is no available proof of the property for a Schr\"odinger operator on a half-line with a dissipative potential and boundary condition, and we provide a proof here (see Theorem \ref{thm:sabc}).
	
	However, in this paper we go beyond maximally dissipative extensions that preserve the differential expression (so-called proper extensions for an appropriate choice of a dual pair, see, e.g.~\cite{Arl12,Fis17, proper,MM97, MM02}). It has long been known that dissipative operators can have non-proper maximally dissipative extensions \cite{CP68}. We refer to \cite{ Nonproper} for more on non-proper extensions; one result of that paper is a characterisation of maximally dissipative extensions for the half-line Schr\"{o}dinger operator which we make use of here, see \eqref{eq:mdoext}. We show 
	that it is possible for half-line Schr\"{o}dinger operators to have non-proper maximally dissipative extensions with a one-dimensional reducing subspace on which the operator is selfadjoint (see Theorems \ref{thm:kv} and \ref{thm:ev}). We give a characterisation of these extensions and the corresponding subspaces and, in the final section, we present a specific example. 
	\section*{Acknowledgements} Part of this work was done by CF and SN at the Institute Mittag-Leffler in Djursholm, Sweden, whose support and hospitality is gratefully acknowledged. It is also a pleasure to thank Fritz Gesztesy for providing useful references and Marius Mitrea for useful discussions and feedback. We are especially indebted to Matthias Hofmann who provided us with the proof of Lemma \ref{lemma:kerinv}, which allowed us to remove some technical conditions on the potential $V$.
	
	\section{Background}\label{sec:bg}
	
	The following notation will be used throughout this paper. We let $\cH$ be a complex Hilbert space with inner product $\langle\cdot,\cdot\rangle$, which we assume to be linear in the second and anti-linear in the first component. A linear operator on $\cH$ will typically be denoted by $A$, its domain by $\cD(A)$ and its range by $\Ran(A)$. We denote the (open) upper and lower half of the complex plane by $\C^+$ and $\C^-$, respectively.
	
	\begin{defn}
		A densely defined linear operator $A$ with domain $\cD(A)$ in $\cH$ is called \textit{dissipative} if $\Im\llangle u,Au\rrangle \geq 0$ for all $u\in \cD(A)$. $A$ is called \textit{anti-dissipative} if $(-A)$ is dissipative.
		Dissipative operators which have no non-trivial dissipative extensions are called \textit{maximally dissipative operators} (MDO).
	\end{defn} 
	
	An MDO consists of one part (corresponding to the set of eigenvectors of the real point spectrum) which looks like a selfadjoint operator and a remaining part. In many situations it makes sense to study the two parts separately. This idea leads to the introduction of the notion of completely non-selfadjoint operators (corresponding to the remaining part of the operator). We first need another definition.
	
	\begin{defn} Let $A$ be an operator on a Hilbert space $\mathcal{H}$. A closed subspace $\mathcal{M}\subset\mathcal{H}$ is called a \emph{reducing subspace} of $A$, or is said to \textit{reduce} the operator $A$, if
		\begin{equation}
			\mathcal{D}(A)=\mathcal{D}(A)\cap\mathcal{M}+\mathcal{D}(A)\cap\mathcal{M}^\perp,
		\end{equation}  
		and if
		\begin{equation}
			A(\mathcal{D}(A)\cap\mathcal{M})\subset\mathcal{M}\quad\mbox{and}\quad A(\mathcal{D}(A)\cap\mathcal{M}^\perp)\subset\mathcal{M}^\perp\:.
		\end{equation}
	\end{defn}
	
	\begin{remark}
		If ${A}$ is an MDO on $\mathcal{H}$, we know that $\C^-\subset\rho(A)$. Instead of working with the possibly unbounded operator $A$, it sometimes proves useful to work with the bounded resolvent $(A-\lambda)^{-1}$ instead. It is well-known that
		a closed subspace $\mathcal{M}$ reduces $A$ if and only if it reduces $(A-\lambda)^{-1}$ for any $\lambda\in\rho(A)$, and we will make use of this throughout the paper.
	\end{remark}
	
	For later purposes, we need the following result on density. 
	\begin{lemma} \label{lemma:spacedense} 
		Let $\mathcal{D}$ be dense in $\mathcal{H}$ and assume that there exists a closed subspace $\mathcal{M}$ such that 
		\begin{equation}
			\mathcal{D}=\mathcal{D}\cap\mathcal{M}\oplus\mathcal{D}\cap\mathcal{M}^\perp\:.
		\end{equation}
		Then, 
		\begin{equation} \label{eq:spaceclos}
			\overline{\mathcal{D}\cap\mathcal{M}}=\mathcal{M}\quad\mbox{and}\quad\overline{\mathcal{D}\cap\mathcal{M}^\perp}=\mathcal{M}^\perp\:.
		\end{equation}
	\end{lemma}
	\begin{proof} We will only show the first identity in \eqref{eq:spaceclos}. Let $m\in\mathcal{M}$ be arbitrary. Since $\mathcal{D}$ is dense, there exists a sequence $\{d_n\}_{n=1}^\infty$ such that $d_n\rightarrow m$. For each $n\in\N$, decompose $d_n=\tilde{d}_n+d_n^\perp$, where $\tilde{d}_n\in\mathcal{D}\cap\mathcal{M}$ and $d_n^\perp\in\mathcal{D}\cap\mathcal{M}^\perp$. Since $\|m-d_n\|^2=\|m-\tilde{d}_n\|^2+\|d^\perp_n\|^2$, this implies that $\tilde{d}_n\rightarrow m$ and $d_n^\perp\rightarrow 0$. Consequently, $\{\tilde{d}_n\}_{n=1}^\infty$ is a sequence in $\mathcal{D}\cap\mathcal{M}$ that converges to $m$, which shows that $\overline{\mathcal{D}\cap\mathcal{M}}=\mathcal{M}$.
	\end{proof}
	
	We next define the property of MDOs that particularly interests us in this paper.
	
	\begin{defn} Let $A$ be an MDO.  $A$ is \textit{completely non-selfadjoint (cns)} if there exists no non-trivial reducing subspace $\cH_1\subseteq \cH$ such that $A\vert_{\cH_1}$ is selfadjoint.
	\end{defn}
	
	The following result gives an explicit formula for the completely non-selfadjoint part of the operator. In the case of a relatively bounded imaginary part the formula is simple. For more general situations the formula involves operators $\Delta$ and $\Delta_*$ which are regularisations of the (possibly non-existing) imaginary part of the operator. 
	
	\begin{prop}\label{prop:Langer} (Langer decomposition, see \cite{Lan61,Nab81}).
		Let $A$ be an MDO. Then there exists a unique decomposition of $\cH=\cH_{sa}(A)\oplus \cH_{cns}(A)$ into an orthogonal sum of two reducing subspaces for $A$ such that $A\vert_{\cH_{sa}(A)}$ is selfadjoint in $\cH_{sa}(A)$ and  $A\vert_{\cH_{cns}(A)}$ is completely non-selfadjoint in $\cH_{cns}(A)$.
		
		Let 
		\be\label{Cayley} T=I-2i(A+i)^{-1} = (A-iI)(A+iI)^{-1}\ee
		be the Cayley transform of $A$,
		define
		\ben\label{Delta}
		\Delta&=& I-T^*T \ =\ 2i\left[(A+i)^{-1}-(A^*-i)^{-1}+2i (A^*-i)^{-1}(A+i)^{-1}  \right],
		\een
		\ben\label{Deltastar}
		\Delta_*&=& I-TT^* \ =\  2i\left[(A+i)^{-1}-(A^*-i)^{-1}+2i (A+i)^{-1}(A^*-i)^{-1} \right]
		\een
		and set
		$$\mathcal{M}:= \Ran(\Delta) +\Ran(\Delta_*)\subseteq \cH.$$
		Then the completely non-selfadjoint part $\cH_{cns}(A)$ is given by the closure of the linear span of $\mathcal{M}$ developed by appropriate resolvents, namely
		$$\cH_{cns}(A) = \clos\left(\mathrm{Span}_{\Im\lambda< 0} \{ (A-\lambda)^{-1}\mathcal{M}\}+\mathrm{Span}_{\Im\lambda> 0} \{ (A^*-\lambda)^{-1}\mathcal{M}\} \right).
		$$
		
		If $A$ has relatively bounded imaginary part, i.e.~$A=L+iV$ with $L=L^*$, $V\geq 0$, $V$ relatively $L$-bounded, then there is a simple explicit expression for the completely non-selfadjoint part $\cH_{cns}(A)$:
		$$\cH_{cns}(A) = \clos\left(\mathrm{Span}_{\Im\lambda\neq 0} \{ (L-\lambda)^{-1}\Ran V\} \right) 
		= \clos\left(\mathrm{Span}_{\lambda\notin(\sigma(A)\cup\R)} \{ (A-\lambda)^{-1}\Ran V\} \right),$$
		i.e.~$\cH_{cns}(A)$ is generated by the range of the imaginary part $V$ developed by the resolvent of the operator $A$ or its real part $L$. Moreover, $A\vert_{\cH_{sa}(A)}=L\vert_{\cH_{sa}(A)}$.
	\end{prop}
	
	For a more explicit description of the completely non-selfadjoint part in the general case which avoids determining the operators $\Delta$ and $\Delta_*$ by making use of the Lagrange identity, see \cite[Theorem 7.6]{BMNW20} and \cite[Theorem 3.3]{BMNW21}.

	We will also need the following characterisation of symmetric operators which are completely non-selfadjoint. 
	
	\begin{prop}[Kre{\u\i}n, \cite{Krein}]\label{prop:Krein} A closed symmetric operator $S$ on a Hilbert space $\cH$ is completely non-selfadjoint if and only if
		\begin{equation} \clos\left(\mathrm{Span}_{\lambda\in\C\setminus\R} \ker(S^*-\lambda) \right)=\mathcal{H}\:.
		\end{equation}
	\end{prop}
	
	Let us also introduce the symmetric subspace of an MDO:
	\begin{defn}
		Let $A$ be an MDO. Then, the \emph{symmetric subspace} $\cH_{sym}(A)$ of $A$ is given by
		\begin{equation} \label{eq:kuzhel}
			\cH_{sym}(A)=\{f\in\cD(A): \forall g\in\cD(A): \langle f,Ag\rangle=\langle Af,g\rangle\}\:.
		\end{equation}	
		
	\end{defn} 
	\begin{remark}\label{rem:kuzhel}
		Clearly, we have $\cH_{sa}(A)\cap\cD(A)\subset \cH_{sym}(A)$. Moreover, in \cite[Prop. 1.1]{Kuz96}, it was shown that 
		\begin{equation}
			\cH_{sym}(A)=\ker(A-A^*)=\{f\in\cD(A)\cap\cD(A^*): Af=A^*f\}\:.
		\end{equation}
		Moreover, extending the definition of the symmetric subspace to maximally anti-dissipative operators in the obvious way, we have that $\cH_{sym}(A)=\cH_{sym}(A^*)$.
	\end{remark}

	\section{The half-line Schr\"odinger operator}
	
	Let $\cH=L^2(\R^+)$ and let $H^2(\R^+)$ denote the usual Sobolev space of order $2$ over $\R^+$. We consider a Schr\"odinger operator on the half-line with dissipative potential. To this end, let the minimal operator $A_{min}$ be given by
	\begin{equation} \label{eq:minimal}
		A_{min}:\quad \cD(A_{min})=\{f\in H^2(\R^+): f(0)=f'(0)=0\},\quad (A_{min}f)(x)=-f''(x)+V(x)f(x)\:,
	\end{equation}
	where $V\in L^\infty(\R^+)$ such that $V_I(x):=\Im V(x)\geq 0$ almost everywhere. 
	For later purposes, it will also be convenient to write $A_{min}=S+iV_I$, where the symmetric operator $S$ is given by
	\begin{equation}
		\cD(S)=\{f\in H^2(\R^+): f(0)=f'(0)=0\},\quad (Sf)(x)=-f''(x)+V_R(x)f(x)\:,
	\end{equation}
	where $V_R(x):=\Re V(x)$ and $V_I$ is the multiplication operator by $V_I(x)$.  Moreover, note that  $S^*$ is given by
	\begin{equation}
		\cD(S^*)=H^2(\R^+),\qquad (S^*f)(x)=-f''(x)+V_R(x)f(x)\:.
	\end{equation}
	Let us now show a useful lemma about the elements of the defect spaces $\ker(S^*-\overline{\lambda})$, where $\lambda\in\mathbb{C}\setminus\mathbb{R}$.
	\begin{lemma} \label{lemma:kerspan}
		Let $\lambda\in\mathbb{C}\setminus\mathbb{R}$. Then, the initial value problem
		\begin{align}
			S^*f=-f''+V_Rf=\overline{\lambda}f, \quad f(0)=1
		\end{align}	
		has a unique solution, which we denote by $\phi_\lambda$.
		Moreover, we have
		\begin{equation}
			\lim_{\mu\rightarrow\lambda}\|\phi_\lambda-\phi_\mu\|=0\:.
		\end{equation}
	\end{lemma}
	\begin{proof}
		Firstly note that $\dim\ker(S^*-\overline{\lambda})=1$ since $S$ is limit-circle at $0$ and limit-point at infinity. Now, note that if $f\in\ker(S^*-\overline{\lambda})$, it is not possible that $f(0)=0$. This follows from the fact that if $f(0)=0$ and $f\in\ker(S^*-\overline{\lambda})$, then $f$ would be in the domain of the selfadjoint operator 
		\begin{equation}
			\hat{S}:\quad\mathcal{D}(\hat{S})=\{f\in H^2(\mathbb{R}^+): f(0)=0\}, \quad f\mapsto-f''+V_Rf\:.
		\end{equation}
		Since in this case, we would have $\hat{S}f=S^*f=\overline{\lambda}f$, this would mean that $f$ is an eigenvector of a selfadjoint operator corresponding to a non-real eigenvalue, which is impossible. Thus, there exists a unique element $\phi_\lambda$ of $\ker(S^*-\overline{\lambda})$ with $\phi_\lambda(0)=1$.
		
		Now, observe that for any $\mu\in\C\setminus\R$, we have
		\begin{equation}
			\phi_\mu=\left[\idty-(\overline{\lambda}-\overline{\mu})(\hat{S}-\overline{\mu})^{-1}\right]\phi_\lambda\:.
		\end{equation}
		To see this, note that it follows from a direct calculation that $\left[\idty-(\overline{\lambda}-\overline{\mu})(\hat{S}-\overline{\mu})^{-1}\right]\phi_\lambda\in\ker(S^*-\overline{\mu})$. Moreover, since $\left[(\hat{S}-\overline{\mu})^{-1}\phi_\lambda\right](0)=0$, the boundary condition at $0$ is satisfied. We therefore get
		\begin{equation}
			\|\phi_\lambda-\phi_\mu\|=|\lambda-\mu|\|(\hat{S}-\overline{\mu})^{-1}\phi_\lambda\|\leq |\lambda-\mu|\frac{1}{|\Imag{\mu}|}\overset{\mu\rightarrow\lambda}{\longrightarrow}0\:,
		\end{equation}
		which finishes the proof.
	\end{proof}
	\begin{remark}
		Note that the proof of the lemma shows that $\C\setminus\R$ is contained in the residual spectrum of the operator $S$: Clearly, for such $\lambda$ the operator is $S-\lambda$ is injective, else $\lambda$ would be a non-real eigenvalue of the symmetric operator $S$; however, the range of $S-\lambda$ is not dense, as its orthogonal complement contains $\phi_\lambda$.
	\end{remark}

	\begin{lemma} \label{lemma:simple}
		The operator $S$ is completely non-selfadjoint.
	\end{lemma}
	\begin{proof}
		
		This follows from \cite[Thm.\ 6.2]{GNWZ18}, where it was shown that 
		\begin{equation}\label{eq:cns}
			\clos\left(\Span\{\phi_\lambda:\lambda\in\C\setminus\R\}\right)=L^2(\R^+)\:.
		\end{equation}
		By Proposition \ref{prop:Krein}, this implies that $S$ is completely non-selfadjoint.
	\end{proof}
	\begin{remark}
		In the literature, a completely non-selfadjoint symmetric operator is commonly referred to as ``simple".
	\end{remark}
	In \cite{Nonproper}, it was shown that all maximally dissipative extensions of $A_{min}$ can be parametrised by $h\in\C^+\cup\R\cup\{\infty\}$ and $k\in\Ran(V_I^{1/2})$ satisfying 
	\begin{equation} \label{eq:disscond}
		\frac{1}{4}\int_\mathcal{E}\frac{|k(x)|^2}{V_I(x)}dx\leq \Im(h)\:  \mbox{ \quad or equivalently\quad }  \frac{1}{4}\|V_I^{-1/2}k\|^2\leq \Imag(h),
	\end{equation}
	where $\mathcal{E}=\{x\in\R^+: V_I(x)\neq 0\}$, which is determined up to a set of Lebesgue measure zero. The extensions are given by
	\begin{align} \label{eq:mdoext}
		A_{h,k}:\qquad \cD(A_{h,k})&=\{f\in L^2(\R^+): f'(0)=hf(0)\}\notag\\
		(A_{h,k}f)(x)&=-f''(x)+V(x)f(x)+f(0)k(x)\:.
	\end{align}
	Note that this means the differential expression $``-\frac{d^2}{dx^2}+V(x)"$ is preserved if and only if $k\equiv 0$.
	Moreover, from a direct calculation it can be seen that the adjoint operator $A_{h,k}^*$ is given by
	\begin{align} \label{eq:mdoadj}
		A_{h,k}^*:\qquad \cD(A_{h,k}^*)&=\{f\in L^2(\R^+): f'(0)=\overline{h}f(0)+\langle k,f\rangle\}\notag\\
		(A_{h,k}^*f)(x)&=-f''(x)+\overline{V(x)}f(x)\:.
	\end{align}
	
	The special case $h=\infty$ corresponds to a Dirichlet boundary condition and Condition \eqref{eq:disscond} simplifies to $k = 0$ a.e.~in this case. In other words, the only maximally dissipative extension of $A_{min}$ which has a Dirichlet boundary condition at $0$ is given by
	\begin{equation} \label{eq:Dirichlet}
		A_{\infty,0}:\quad\cD(A_{\infty,0})=\{f\in H^2(\R^+): f(0)=0\}, \qquad (A_{\infty,0}f)(x)=-f''(x)+V(x)f(x)\:.
	\end{equation}

	Next, let us introduce the subspace 
	\begin{equation*}
		\cD_0:=\cD(A_{min})\cap\ker(V_I)\:.
	\end{equation*}
	Note that for any $f\in\cD_0$ we have $A_{h,k}f=Sf$. Thus, for any $f\in\cD_0$ and $g\in\cD(A_{h,k})$, we get
	\begin{align}
		\langle f,A_{h,k}g\rangle&=\langle f,-g''+V_Rg+iV_Ig+g(0)k\rangle=\langle f,-g''+V_Rg\rangle=\langle -f''+V_Rf,g\rangle\notag
		\\&=\langle Sf,g\rangle=\langle A_{h,k}f,g\rangle\:,
	\end{align}
	where we used $f(0)=f'(0)=0$ as well as $f\in\ker(V_I)$, which implies, since $k\in\Ran(V_I^{1/2})$, that we have $f\perp k$. Hence, it is always true that $\cD_0\subset\cH_{sym}(A_{h,k})$. If we assume in addition that $\cH_{sym}(A_{h,k})\subset\cD_0$ -- or equivalently that $\cH_{sym}(A_{h,k})=\cD_0$ -- let us show that $A_{h,k}$ is completely non-selfadjoint:
	
	\begin{lemma} \label{lemma:D0}
		If $\mathcal{H}_{sym}(A_{h,k})=\mathcal{D}_0$, then $A_{h,k}$ is completely non-selfadjoint. 
	\end{lemma}
	\begin{proof}
		Let $\eta\in\mathcal{H}_{sa}(A_{h,k})$. For any $\lambda\in\C\setminus\R$, let $\phi_\lambda$ be such that  $\ker(S^*-\overline{\lambda})=\Span\{\phi_\lambda\}$. For any $\lambda\in\C^-$, we get
		\begin{align} \label{eq:ortho}
			\langle\phi_\lambda,\eta\rangle=\langle\phi_\lambda,(A_{h,k}-\lambda)(A_{h,k}-\lambda)^{-1}\eta\rangle=\langle\phi_\lambda,(S-\lambda)(A_{h,k}-\lambda)^{-1}\eta\rangle=0\:,
		\end{align}  	
		where we used that by assumption $(A_{h,k}-\lambda)^{-1}\eta\in\mathcal{D}(A_{h,k})\cap\mathcal{H}_{sa}(A_{h,k})\subset\mathcal{H}_{sym}(A_{h,k})=\mathcal{D}_0$. By a similar argument, replacing $A_{h,k}$ by $A_{h,k}^*$, we find that $\langle \phi_\lambda,\eta\rangle=0$ for all $\lambda\in\C^+$. By \eqref{eq:cns}, this implies that $\eta=0$ and thus $A_{h,k}$ is completely non-selfadjoint by Proposition \ref{prop:Krein}.
	\end{proof}
	
	\subsection{The non-critical case} We refer to the case when we have strict inequality in \eqref{eq:disscond} as the non-critical case.
	We are ready to prove our first main result:
	\begin{thm} \label{thm:noncrit}
		If $h\in\C^+$ and $k\in\Ran(V_I^{1/2})$ are such that we have strict inequality in \eqref{eq:disscond}, then $A_{h,k}$ is completely non-selfadjoint.
	\end{thm}
	\begin{proof}  Firstly, for any $f\in\mathcal{D}(A_{h,k})$, consider
		\begin{align*}
			\Im\langle f,A_{h,k}f\rangle &=\Im\langle f,-f''\rangle +\|V_I^{1/2}f\|^2+\Im\langle f,f(0)k\rangle\\
			&=\Im(h)|f(0)|^2+\|V_I^{1/2}f\|^2+\Im\langle V_I^{1/2}f,f(0)V_I^{-1/2}k\rangle\\
			&=\left(\Im(h)-\frac{1}{4}\|V_I^{-1/2}k\|^2\right)|f(0)|^2+\left\|V_I^{1/2}f- i\frac{f(0)}{2}V_I^{-1/2}k\right\|^2\\
			&\geq \left(\Im(h)-\frac{1}{4}\|V_I^{-1/2}k\|^2\right)|f(0)|^2\:.
		\end{align*}
		Now, if $f\in\cH_{sym}(A_{h,k})$, this implies in particular $\Im\langle f,A_{h,k}f\rangle=0$. Consequently, if $f\in \cH_{sym}(A_{h,k})$, we get
		\begin{equation*}
			0=\Imag\langle f,A_{h,k}f\rangle\geq \left(\Im(h)-\frac{1}{4}\|V_I^{-1/2}k\|^2\right)|f(0)|^2\:.
		\end{equation*}
		Since $\Im(h)>\frac{1}{4}\|V_I^{-1/2}k\|^2$ by assumption, this implies $f(0)=0$ and consequently, $f\in\cD(A_{min})$. We then get
		\begin{equation}
			0=\Im\langle f,A_{h,k}f\rangle=\|V_I^{1/2}f\|^2\:,
		\end{equation}
		and thus $f\in\ker(V_I^{1/2})=\ker(V_I)$. This shows that $\cH_{sym}(A_{h,k})=\cD_0$, which by Lemma \ref{lemma:D0} implies that $A_{h,k}$ is completely non-selfadjoint.
	\end{proof}
	\begin{remark}
		With this result, the case of a purely real potential $V$, corresponding to $V_I\equiv 0$ is completely covered. In this situation, all maximally dissipative extensions of $A_{min}$ are given by $A_{h,0}$, where $h\in\C^+\cup\R\cup\{\infty\}
		$. Now, if $\Im(h)>0$, we can apply Theorem \ref{thm:noncrit} and conclude that $A_{h,0}$ is completely non-selfadjoint. On the other hand, if $h=\infty$ or $\Im(h)=0$, the operator $A_{h,0}$ is obviously selfadjoint. Thus, in the following, we will always assume that $V_I\not\equiv 0$.   
	\end{remark}
	
	\subsection{The critical case}
	Let us now focus on the critical case $\Im(h)=\frac{1}{4}\|V_I^{-1/2}k\|^2$, which we will assume for the remainder of this paper. We will have to distinguish two cases: (i) the case of a selfadjoint boundary condition, corresponding to $\frac{1}{4}\|V_I^{-1/2}k\|^2=\Im(h)=0$ and (ii) the case of a dissipative boundary condition, corresponding to $\frac{1}{4}\|V_I^{-1/2}k\|^2=\Im(h)>0$.

	\subsubsection{Selfadjoint boundary condition}
	Since we are in the critical case $\frac{1}{4}\|V_I^{-1/2}k\|^2=\Im(h)$ but also assume a selfadjoint boundary condition $\Im(h)=0$, this implies that $k= 0$ a.e. Hence, we are considering only operators of the form
	\begin{equation} A_{h,0}:\quad\cD(A_{h,0})=\{f\in H^2(\R^+): f'(0)=hf(0)\},\qquad (A_{h,0}f)(x)=-f''(x)+V(x)f(x)\:,
	\end{equation}
	where $h\in\R\cup\{\infty\}$. Let us now show that all these operators are completely non-selfadjoint. As usual, $h=\infty$ is interpreted as the Dirichlet condition at the endpoint.

	\begin{thm} \label{thm:sabc}
		Assume $V_I\not\equiv 0$. The operators $A_{h,0}$, where $h\in\R\cup\{\infty\}$ are completely non-selfadjoint.
	\end{thm}
	
	\begin{proof}
		Let $|M|$ denote the Lebesgue measure of a set $M\subset\R^+$ and define
		$$x_0:= \inf \left\{x>0: \left\vert(0,x+\eps)\cap\{x\in\R^+: V_I(x)>0\}\right\vert>0 \hbox{ for all } \eps>0\right\}.$$
		Then $x_0\geq 0$ and we have $V_I(x) = 0 $ a.e.~on $(0,x_0]$.
		
		Choose a decreasing sequence $(x_n)$ in $(x_0,\infty)$ such that $x_n\to x_0$. Then
		$$\left\vert(x_0,x_n+\eps)\cap\{x\in\R^+: V_I(x)>0\}\right\vert>0 \hbox{ for all } \eps>0.$$
		Let $A_s\subset A_{h,0}$ such that $A_s$ is selfadjoint on the subspace $\overline {\cD(A_s)}\subseteq \cH$. We need to show that $\cD(A_s)=\{0\}$. Let $u\in \cD(A_s)$, then
		$$0=\Im \llangle u, A_s u\rrangle \geq \int_{x_0}^{x_n+\eps} V_I |u|^2.$$
		Thus, for any $\eps>0$, we can choose a sequence $(\tilde{x}_n)$ with $\tilde{x}_n\in (x_0,x_n+\eps)$ such that $u(\tilde{x}_n)=0$. As $u\in H^2(\R^+)$, it is continuous, so $u(x_0)=0$. 
		
		Moreover, $u'$ is continuous, so both its real and imaginary parts are. By the Mean Value Theorem there exist  a sequence $(x_n')$  with $\tilde{x}_{n+1}\leq x_n'\leq \tilde{x}_n$ such that $\Re(u')(x_n')=0$ and another sequence $(x_n'')$  with $\tilde{x}_{n+1}\leq x_n''\leq \tilde{x}_n$ such that $\Im(u')(x_n'')=0$. By continuity, $\Re(u')(x_0)=0=\Im(u')(x_0)$, so $u'(x_0)=0$. 
		
		First assume $x_0>0$ and consider the operator $A_{x_0}:=A_{(0,x_0)}\oplus A_{(x_0,\infty)}$, where both $A_{(0,x_0)}$ and $A_{(x_0,\infty)}$ are given by the expression 
		$$f\mapsto - f''+V_R f
		$$	 	
		with
		$$\cD(A_{(0,x_0)}):=\{f\in H^2(0,x_0):   f'(0)=hf(0), f(x_0)=0 \}$$
		and $$\cD(A_{(x_0,\infty)}):=\{f\in H^2(x_0,\infty): f(x_0)=0 \}.$$
		Since $u(x_0)=0$, we have that $u_0:=u\vert_{(0,x_0)} \in \cD(A_{(0,x_0)})$ and $u_\infty=u\vert_{(x_0,\infty)} \in \cD(A_{(x_0,\infty)})$.
		
		Next, let $\lambda\in\C^-$ and consider $(A_s-\lambda)u=g$ for $u\in \cD(A_s)$ and $g\in\overline{\cD(A_s)}$. Then
		$$(A_s-\lambda)u=g \Longleftrightarrow - u''+(V-\lambda) u=g\Longleftrightarrow - u''+(V_R-\lambda) u=g \Longleftrightarrow (A_{x_0}-\lambda)u=g$$
		and
		$$u=(A_{x_0}-\lambda)^{-1}g= u_0\oplus u_\infty=(A_{(0,x_0)}-\lambda)^{-1} g_0 \oplus (A_{(x_0,\infty)}-\lambda)^{-1} g_\infty,$$
		where $g_0=g\vert_{(0,x_0)}$ and $g_\infty=g\vert_{(x_0,\infty)}$.
		
		Let $G_{(0,x_0)}$ denote the Green's function associated with $A_{(0,x_0)}$ and let $\varphi_l,\varphi_r$ be the solutions to 
		$-\varphi''+V_R\varphi=\lambda\varphi$ in $(0,x_0)$ satisfying $\varphi_l'(0,\lambda)=h\varphi_l(0,\lambda)$ and $\varphi_r(x_0,\lambda)=0$, respectively. Let $W(\lambda)$ be the corresponding Wronskian. Then for $x\in(0,x_0)$ we have
		\begin{eqnarray*}
			u_0(x)&=& \int_0^{x_0}G_{(0,x_0)}(x,y)g_0(y) dy \\ & = & \int_0^x \frac{\varphi_r(x,\lambda)\varphi_l(y,\lambda)}{W(\lambda)}g_0(y)\ dy + \int_x^{x_0} \frac{\varphi_l(x,\lambda)\varphi_r(y,\lambda)}{W(\lambda)}g_0(y)\ dy.
		\end{eqnarray*}
		Moreover,
		\begin{eqnarray}
			u_0'(x)&=& \int_0^x \frac{\varphi_r'(x,\lambda)\varphi_l(y,\lambda)}{W(\lambda)}g_0(y)\ dy + \int_x^{x_0} \frac{\varphi_l'(x,\lambda)\varphi_r(y,\lambda)}{W(\lambda)}g_0(y)\ dy.
		\end{eqnarray}
		As shown above, since $u\in \cD(A_s)$,
		\begin{equation}\label{eq:uprime}
			u_0'(x_0)= \frac{\varphi_r'(x_0,\lambda)}{W(\lambda)} \int_0^{x_0} \varphi_l(y,\lambda)g_0(y)\ dy =0.
		\end{equation}
		Since $\varphi_r'(x_0,\lambda)\neq 0$, Equation \eqref{eq:uprime} implies that $\overline{g}_0\perp\varphi_l(\cdot,\lambda)$  for all $\lambda\in\C^-$. By analyticity, $\overline{g}_0\perp\varphi_l(\cdot,\lambda)$  for all $\lambda\in\C$. Choosing $\lambda$ to run through all eigenvalues of $A_{(0,x_0)}$, we get that $\overline{g}_0$ is orthogonal to all eigenvectors and root vectors. These are complete, see e.g.~\cite[Chapter V.2]{GK69} or \cite{WW12}, so therefore, $g_0(x)=0$.
		
		Next we consider $g_\infty$. Let $\tilde{\varphi}_l$  be the solution to 
		$-\varphi''+V_R\varphi=\lambda\varphi$ in $(x_0,\infty)$ satisfying $\tilde{\varphi}_l(x_0,\lambda)=0$. 
		Let $f$ be the  $L^2$-solution of $- f''+V_R f=\lambda f$
		and $\tilde W(\lambda)$ be the corresponding Wronskian. Then for $x\in(x_0,\infty)$ we have
		\begin{eqnarray}
			u_\infty(x)&=& \int_{x_0}^x \frac{f(x,\lambda)\tilde\varphi_l(y,\lambda)}{\tilde W(\lambda)}g_\infty(y)\ dy + \int_x^\infty \frac{\tilde\varphi_l(x,\lambda)f(y,\lambda)}{\tilde W(\lambda)}g_\infty(y)\ dy.
		\end{eqnarray}
		Moreover,
		\begin{eqnarray}
			u_\infty'(x)&=& \int_{x_0}^x \frac{f'(x,\lambda)\tilde\varphi_l(y,\lambda)}{\tilde W(\lambda)}g_\infty(y)\ dy + \int_x^\infty \frac{\tilde\varphi_l'(x,\lambda)f(y,\lambda)}{\tilde W(\lambda)}g_\infty(y)\ dy.
		\end{eqnarray}
		Since $u\in \cD(A_s)$,
		\begin{equation}\label{eq:uprime2}
			u_\infty'(x_0)= \frac{\tilde\varphi_l'(x_0,\lambda)}{\tilde W(\lambda)} \int_{x_0}^\infty f(y,\lambda)g_\infty(y)\ dy =0.
		\end{equation} 
		This implies that $\overline{g_\infty}\perp f(\cdot,\lambda)$  for all $\lambda\in\C\setminus\R$. By  \cite[Theorem 6.2]{GNWZ18} we get $g_\infty(x)=0$. 
		
		Therefore, $g(x)=g_0(x)+g_\infty(x)=0$ a.e.~and this implies that $u=0$ and so $\cD(A_s)$ is trivial.
		
		Finally, we note that if $x_0=0$, then $u=u_\infty$ which can be shown to be zero also in this case by the same argument as above.
	\end{proof}
	
	\subsubsection{Dissipative boundary condition}	
	
	In this section, we will only consider the case where $\frac{1}{4}\|V_I^{-1/2}k\|^2=\Im(h)>0$. We investigate when the operators $A_{h,k}$ are completely non-selfadjoint and when they possess a non-trivial reducing selfadjoint subspace. Since we are only considering the selfadjoint/completely non-selfadjoint/symmetric subspaces of the operator $A_{h,k}$, we will drop the dependence on $A_{h,k} $ and only write $\cH_{sa}, \cH_{cns}$ and $\cH_{sym}$, respectively.
	
	The next two lemmas give a more explicit description of $\cH_{sym}$.
	
	\begin{lemma} \label{lemma:hsymmg}
		We have $\cH_{sym}=\ker(A_{h,k}-A_{h,k}^*)=\{f\in\mathcal{D}(A_{h,k}): V_If=\frac{i}{2}f(0)k\}$.
	\end{lemma}
	\begin{proof}
		The first equality was established in Remark \ref{rem:kuzhel}, so we only need to show the second.
		It is obvious that $\ker(A_{h,k}-A_{h,k}^*)\subset \{f\in\mathcal{D}(A_{h,k}): V_If=\frac{i}{2}f(0)k\}$. Hence, we need to show that if $f\in\mathcal{D}(A_{h,k})$ such that $V_If=\frac{i}{2}f(0)k$, then this implies that $f\in\mathcal{D}(A_{h,k}^*)=\{f\in H^2(\R^+): f'(0)=\overline{h}f(0)+\langle k,f\rangle\}$, see also \eqref{eq:mdoadj}. But this follows from
		\begin{align}
			f'(0)&=hf(0)=\overline{h}f(0)+2i\Im(h)f(0)=\overline{h}f(0)+\left\langle V_I^{-1/2}k,V_I^{-1/2}\left(\frac{i}{2}f(0)k\right)\right\rangle\\
			&=\overline{h}f(0)+\langle k,f\rangle\:,
		\end{align}
		where we have used that $\Im(h)=\frac{1}{4}\|V_I^{-1/2}k\|^2$ and $V_If=\frac{i}{2}f(0)k$.
	\end{proof}

	\begin{lemma} \label{lemma:2cases} We have the following two cases:
		\begin{itemize} 
			\item[(i)]
			If there exists a function $K_V\in\mathcal{D}(A_{h,k})$ with $K_V(0)\neq 0$ such that $V_IK_V=\frac{i}{2}K_V(0)k$, then $\mathcal{H}_{sym}=\mathcal{D}_0\dot{+}\spann\{K_V\}$.
			\item[(ii)] If there does not exist such a function, then 
			$\mathcal{H}_{sym}=\mathcal{D}_0$\:.
		\end{itemize}
	\end{lemma}
	\begin{proof}
		By Lemma \ref{lemma:hsymmg}, we have $\mathcal{H}_{sym}=\ker(A_{h,k}-A_{h,k}^*)=\{f\in\mathcal{D}(A_{h,k}): V_If=\frac{i}{2}f(0)k\}$. 
		Case (i): clearly, we have $\spann\{K_V\}\dot{+}\mathcal{D}_0\subset \mathcal{H}_{sym}$. Let us now show the other inclusion. Let $f\in\mathcal{H}_{sym}=\ker(A_{h,k}-A_{h,k}^*)$. If $f(0)=0$, then $f\in\mathcal{H}_{sym}$ implies $V_If=\frac{i}{2}f(0)k=0$ and thus $f\in\mathcal{D}_0$. Hence, assume $f(0)\neq 0$ from now on. By decomposing
		\begin{equation}
			f=f-\frac{f(0)}{K_V(0)}K_V+\frac{f(0)}{K_V(0)}K_V,
		\end{equation}
		it is sufficient to show that $f-\frac{f(0)}{K_V(0)}K_V\in\cD_0$.
		This follows by observing that 
		\begin{equation}
			V_I\left(f-\frac{f(0)}{K_V(0)}K_V\right)=\frac{i}{2}f(0)k-\frac{i}{2}\frac{f(0)}{K_V(0)}K_V(0)k=0\:,
		\end{equation}
		and using that $f'(0)=hf(0)$ and $K_V'(0)=hK_V(0)$ since $f,K_V\in D(A_{h,k})$. Thus, we have shown that $f\in\mathcal{D}_0\dot{+}\spann\{K_V\}$.
		
		Case (ii): Again, it is clear that $\mathcal{D}_0\subset\mathcal{H}_{sym}$. Now, let $f\in\mathcal{H}_{sym}$. If $f(0)\neq 0$, this would contradict the assumptions of Case (ii) since one could choose $K_V:=f$ in this case. Thus, since $f(0)=0$ and since $f\in\mathcal{H}_{sym}$, we get $V_If=\frac{i}{2}f(0)k=0$ and thus $f\in\mathcal{D}_0$. This shows the lemma.
	\end{proof}
	
	Before we proceed, we need to prove the following technical result.
	\begin{lemma}\label{lemma:kerinv} Suppose that $g\in H^1(\R^+)$ and $V$ is measurable. If $Vg=0$ a.e., then this implies that $Vg'=0$ a.e.	
	\end{lemma}
\begin{proof} Let $B:=\{x\in\R^+: V(x)=0\}$; this is defined up to a null set $\mathcal{K}$. Since $Vg=0$ a.e., there is a null set $\mathcal{N}$ such that 
$g(x)=0$ if $x\in\R^+\setminus(B\cup \mathcal{N}\cup \mathcal{K})$. Then, up to a countable set $I$, no points in $\R^+\setminus(B\cup \mathcal{N}\cup \mathcal{K})$ are isolated. So, for every $x\in \R^+\setminus(B\cup \mathcal{N}\cup \mathcal{K}\cup I)$, we can find a sequence $\{x_n\}\subset\R^+\setminus(B\cup \mathcal{N}\cup \mathcal{K}\cup I)$, where $x_n\neq x$, such that $x_n\rightarrow x$.
Since $g\in H^1(\R^+)$, we have that $g$ is differentiable -- up to a null set $F$. Consequently, we have 
\begin{equation}
	g'(x)=\lim_{n\rightarrow\infty}\frac{g(x)-g(x_n)}{x-x_n}=0
\end{equation}
for every $x\in\R^+\setminus(B\cup\mathcal{N}\cup\mathcal{K}\cup I\cup F)$ and hence, we conclude that
$V(x)g'(x)=0$ for every $x\in\R^+\setminus(\mathcal{N}\cup\mathcal{K}\cup I\cup F)$. But since $\mathcal{N}, \mathcal{K}, I,$ and $F$ are all null sets, this implies that $Vg'=0$ a.e. and thus the lemma.
\end{proof}
	We are now ready to prove a uniqueness result in $\cH_{sa}$.
	
	\begin{lemma} \label{lemma:resolvent}
	 Let $f\in\mathcal{H}_{sa}$ and assume there exists a $\lambda\in\C^-$ such that $(A_{h,k}-\lambda)^{-1}f\in\mathcal{D}_0$. Then $f=0$.
	\end{lemma}
	\begin{proof}  Take any $\mu\in\C^-$. By the resolvent identity, we get
		\begin{equation} \label{eq:resf}
			(A_{h,k}-\mu)^{-1}f=(A_{h,k}-\lambda)^{-1}f+(\mu-\lambda)(A_{h,k}-\mu)^{-1}(A_{h,k}-\lambda)^{-1}f\:.
		\end{equation}
		By assumption, we have $g:=(A_{h,k}-\lambda)^{-1}f\in\mathcal{D}_0$. Let us thus focus on the second term $(A_{h,k}-\mu)^{-1}g$. Using that $(A_{h,k}-\mu)^{-1}g\in\mathcal{H}_{sa}\cap\mathcal{D}(A_{h,k})\subset \mathcal{H}_{sym}$, we know from Lemma \ref{lemma:2cases} that there exist $g_0\in\mathcal{D}_0$ and $\tau\in\C$ such that
		\begin{equation} \label{eq:negativeimaginary}
			(A_{h,k}-\mu)^{-1}g=g_0+\tau K_V    
		\end{equation}
		or -- equivalently -- 
		\begin{align} 
			g&=(A_{h,k}-\mu)(A_{h,k}-\mu)^{-1}g \nonumber \\
			&=-g_0''+V_Rg_0-\mu g_0+\tau\left( -K_V''+V_RK_V+iV_IK_V+K_V(0)k-\mu K_V \right) \nonumber \\  &=-g_0''+V_Rg_0-\mu g_0+\tau\left(-K_V''+V_RK_V-iV_IK_V-\mu K_V\right)\:. \label{eq:ginD0}
		\end{align}
		Now, observe that since $g,g_0\in\cD_0$ we have $g\in\ker(V_I)$ as well as $-g_0''+V_Rg_0-\mu g_0\in\ker(V_I)$, which follows from a twofold application of Lemma \ref{lemma:kerinv} to $g_0$. Next, let us consider $-K_V''+V_RK_V-iV_IK_V-\mu K_V$. We want to show that there exists at most one $\tilde{\mu}\in\C^-$ such that $-K_V''+V_RK_V-iV_IK_V-\tilde{\mu} K_V\in\ker(V_I)$. To this end, assume the following equality holds
		\begin{equation} \label{eq:notkern}
			\langle V_IK_V, -K_V''+V_RK_V-iV_IK_V-\mu K_V\rangle=0. 
		\end{equation}
		Then,
		\begin{eqnarray} \label{eq:notkern2}
			0&=& \langle V_IK_V, -K_V''+V_RK_V-iV_IK_V-\mu K_V\rangle\\
			&=& -\langle V_IK_V,K_V''+V_RK_V\rangle-i\|V_IK_V\|^2-\mu\|V_I^{1/2}K_V\|^2. \nonumber
		\end{eqnarray}
		Since $K_V\notin\ker(V_I)$, we have $\|V_I^{1/2}K_V\|^2>0$ and thus -- depending on whether 
		\begin{equation}
			\tilde{\mu}:=\frac{-\langle V_IK_V,K_V''+V_RK_V\rangle-i\|V_IK_V\|^2}{\|V_I^{1/2}K_V\|^2}
		\end{equation}
		is an element of $\C^-$ or not --
		there exists at most one solution $\tilde{\mu}\in\C^-$ such that \eqref{eq:notkern} is satisfied. Hence, if $\mu\in\C^-\setminus\{\tilde{\mu}\}$, we have $\langle V_IK_V, -K_V''+V_RK_V-iVK_V-\mu K_V\rangle\neq 0$ and thus $-K_V''+V_RK_V-iV_IK_V-\mu K_V\notin\ker(V_I)$. So, if $\mu\neq\tilde{\mu}$, then Equation \eqref{eq:ginD0} implies that $\tau=0$. Consequently, $(A_{h,k}-\mu)^{-1}g=(A_{h,k}-\mu)^{-1}(A_{h,k}-\lambda)^{-1}f=g_0\in\mathcal{D}_0$ and \eqref{eq:resf} therefore implies that $(A_{h,k}-\mu)^{-1}f\in\mathcal{D}_0$ for all $\mu\in\C^-\setminus\{\tilde{\mu}\}$. Letting $\phi_\mu$ be as in the proof of Lemma \ref{lemma:kerspan}, we get by the same argument as in \eqref{eq:ortho} that $\langle\phi_\mu,f\rangle=0$ for all $\mu\in\C^-\setminus\{\tilde{\mu}\}$.
		%
		
		Next, since $f\in \cH_{sa}$, we have $g\in \cH_{sa}$ and therefore, $(A_{h,k}^*-\overline{\mu})^{-1}g \in \cH_{sa}\cap \cD(A_{h,k}^*)$. Now, $\cH_{sa}\cap \cD(A_{h,k}^*)\subset \cH_{sym}(A_{h,k}^*)=\cH_{sym}$. Thus, there exist $\hat{g}_0\in\mathcal{D}_0$ and $\hat{\tau}\in\C$ such that $(A_{h,k}^*-\overline{\mu})^{-1}g=\hat{g}_0+\hat{\tau}K_V$. By repeating the same argument as was presented after Equation \eqref{eq:negativeimaginary}, we may conclude that there exists at most one $\hat{\mu}\in\C^+$ such that $(A^*_{h,k}-\overline{\mu})^{-1}f\in\mathcal{D}_0$ for all $\overline{\mu}\in\C^+\setminus\{\hat{\mu}\}$. Consequently, we have shown
		\begin{equation}
			\langle \phi_\mu,f\rangle=0
		\end{equation}
		for all $\mu\in\C\setminus(\R\cup\{\tilde{\mu},\overline{\hat{\mu}}\})$. By Lemma \ref{lemma:kerspan}, we have 
		\begin{equation}
			0\leq|\langle \phi_{\tilde\mu},f\rangle|=|\langle \phi_{\tilde\mu},f\rangle-\langle \phi_\mu,f\rangle|\leq \| \phi_{\tilde\mu}-\phi_\mu\|\|f\|\overset{\mu\rightarrow\tilde{\mu}}{\longrightarrow}0\:.
		\end{equation}
		Analogously, it also follows that $\langle \phi_{\overline{\hat{\mu}}},f\rangle=0$. Arguing similarly as in Lemma \ref{lemma:D0}, we get $f=0$, which finishes the proof.
	\end{proof}
	
	\begin{cor} \label{coro:sadim}
 If $\mathcal{H}_{sym}=\mathcal{D}_0\dot{+}\spann\{K_V\}$, then $\dim\mathcal{H}_{sa}\leq 1$.
	\end{cor}
	\begin{proof} Fix an arbitrary $\lambda\in\C^-$. Now, assume there exist two linearly independent $f, g\in\mathcal{H}_{sa}$. Since $\cH_{sa}\cap \cD(A_{h,k})\subset \cH_{sym}$, we know that there exist $f_0,g_0\in\mathcal{D}_0$ and $\tau_f,\tau_g\in\C\setminus\{0\}$ such that
		\begin{equation}
			(A_{h,k}-\lambda)^{-1}f=f_0+\tau_f K_V  \mbox{ and }  (A_{h,k}-\lambda)^{-1}g=g_0+\tau_g K_V\:.
		\end{equation}
		This implies
		\begin{equation}
			(A_{h,k}-\lambda)^{-1}(\tau_g f-\tau_fg)= \tau_g f_0-\tau_f g_0\in\mathcal{D}_0\:, 
		\end{equation}
		which by Lemma \ref{lemma:resolvent} implies $\tau_g f-\tau_fg=0$ and hence that $f$ and $g$ are linearly dependent.
	\end{proof}
	
	We are now prepared to show the main result for the critical case with dissipative boundary conditions. It turns out that for the operator $A_{h,k}$ to have a non-trivial reducing selfadjoint subspace, a few rather restrictive conditions have to be met:
	
	\begin{thm}\label{thm:kv} Suppose that equality holds in \eqref{eq:disscond}. The maximally dissipative extension $A_{h,k}$ has a non-trivial reducing selfadjoint subspace if and only if there exists a function $K_V\in\mathcal{D}(A_{h,k})$ with $K_V(0)\neq 0$ such that
		\begin{itemize}
			\item[(i)] $V_IK_V=\frac{i}{2}K_V(0)k$ and
			\item[(ii)] $A_{h,k}K_V
			\in\spann\{K_V\}$\:.
		\end{itemize}
		Moreover, if both Conditions (i) and (ii) are satisfied, then $\mathcal{H}_{sa}=\spann\{K_V\}$.
		\label{thm:cnsa}
	\end{thm}
	\begin{proof}
		Let us first assume that $A_{h,k}$ has a non-trivial reducing selfadjoint subspace. Combining Lemma \ref{lemma:D0} with Lemma \ref{lemma:2cases}, this implies that there exists a function $\tilde{K}_V$ with $\tilde{K}_V(0)\neq 0$ such that $V_I\tilde{K}_V=\frac{i}{2}\tilde{K}_V(0)k$ and that $\cH_{sym}=\cD_0\dot{+}\spann\{\tilde{K}_V\}$. 
		By Corollary \ref{coro:sadim}, we know that $\dim\cH_{sa}= 1$. As this is finite-dimensional, by Lemma \ref{lemma:spacedense}, we have $\cH_{sa}\cap\cD(A_{h,k})=\cH_{sa}$. Consequently, we have $\cH_{sa}\subset\cH_{sym}=\cD_0\dot{+}\spann\{\tilde{K}_V\}$. Now, it is not possible that $\cH_{sa}=\spann\{f_0\}$ for some $f_0\in\cD_0$, since we would have $(A_{h,k}+i)^{-1}f_0\in\cD_0$ and thus by Lemma \ref{lemma:resolvent}, we would have $f_0=0$. Hence, $\cH_{sa}=\spann\{f_0+\tilde{K}_V\}$ for some $f_0\in\cD_0$. Set $K_V=f_0+\tilde{K}_V$. Since $f_0\in\cD_0$, the function $K_V$ satisfies $K_V(0)\neq 0$ as well as Condition (i). Since $\cH_{sa}$ is a reducing subspace, $A_{h,k}K_{V}\in\spann\{K_{V}\}$, as required.
		
		Let us now assume that there exists a $K_V\in\cD(A_{h,k})$ with $K_V(0)\neq 0$ such that Conditions (i) and (ii) are both satisfied. From Condition (ii), we get $A_{h,k}K_V=zK_V$ for some $z\in\C$. Moreover, since $K_V\in\cH_{sym}$, we have $A_{h,k}K_V=A^*_{h,k}K_V=zK_V$ and therefore $\spann\{K_V\}$ reduces $A_{h,k}$. This also implies that $z=\overline{z}\in\R$ and thus $\cH_{sa}=\spann\{K_V\}$.
	\end{proof}

	We summarize the results shown so far in this chapter:
	\begin{enumerate}
		\item 	$\Im h > 0$ :
		\begin{enumerate}
			\item $\Im(h)>\frac{1}{4}\|V_I^{-1/2}k\|^2$:  $A_{h,k}$ is completely non-selfadjoint;
			\item $\Im(h)=\frac{1}{4}\|V_I^{-1/2}k\|^2$:
			\begin{enumerate}
				\item the conditions (i) and (ii) in Theorem \ref{thm:kv} are not satisfied: $A_{h,k}$ is completely non-selfadjoint;
				\item the conditions (i) and (ii) in Theorem \ref{thm:kv} are  satisfied: $A_{h,k}$ has a non-trivial
				reducing selfadjoint subspace $\mathcal{H}_{sa}=\spann\{K_V\}$;
			\end{enumerate}
		\end{enumerate}
		\item  $h\in\R\cup\{\infty\}$:
		\begin{enumerate}
			\item $V_I\equiv 0$:  $A_{h,k}$ is selfadjoint;
			\item $V_I\not\equiv 0$: $A_{h,k}$ is completely non-selfadjoint.
		\end{enumerate}
	\end{enumerate}

	\subsection{Construction of maximally dissipative extensions with a real eigenvalue}
	
	\begin{thm}\label{thm:ev} Let the operator $A_{min}$ be given by \eqref{eq:minimal} and let $\lambda\in\R$. If there exists a non-zero solution $g\in H^2(\R^+)$ to the differential equation 
		\begin{equation} \label{eq:eivallambda}
			A_{min}^*g=-g''+V_Rg-iV_Ig=\lambda g\:,
		\end{equation}	
		then there exists a unique maximally dissipative extension $A_{\lambda}$ of $A_{min}$ for which $\ker(A_\lambda-\lambda)=\mbox{span}\{g\}$ is a non-trivial reducing selfadjoint subspace. \label{prop:1}
	\end{thm}
	\begin{proof} Let $g\in H^2(\R^+)$ be a non-trivial solution of \eqref{eq:eivallambda}. Multiplying both sides of \eqref{eq:eivallambda} by $\overline{g}$ and integrating from $0$ to $\infty$ yields -- using integration by parts --
		\begin{equation} \label{eq:1}
			\overline{g(0)}g'(0)=\int_0^\infty (-V_R(x)+iV_I(x)+\lambda)|g(x)|^2\mbox{d}x-\int_0^\infty|g'(x)|^2\mbox{d}x\:.
		\end{equation}
		Let us now argue that $\Imag\overline{g(0)}g'(0)\neq 0$. By way of contradiction, assume that $\Imag\overline{g(0)}g'(0)=0$. Then by comparing the imaginary parts in \eqref{eq:1}, we get
		\begin{equation}
			\int_0^\infty V_I(x)|g(x)|^2\mbox{d}x=0\:,
		\end{equation}
		which is equivalent to $g\in\ker(V_I)$. Now, since $\Imag\overline{g(0)}g'(0)=0$, we have $g'(0)=hg(0)$ for some $h\in\R\cup\{\infty\}$. Then $g\in\cD(A_{h,0})=\cD(A_{h,0}^*)$ -- the maximally dissipative extension of $A_{min}$ with the selfadjoint boundary condition $g'(0)=hg(0)$. Since $g\in\ker(V_I)$, we get
		\begin{equation}
			A_{h,0}g=-g''+V_Rg+iV_Ig=-g''+V_Rg-iV_Ig=A_{h,0}^*g=A_{min}^*g=\lambda g\:.
		\end{equation}
		But this would mean that $g\in\cH_{sa}(A_{h,0})$ contradicting the fact that $A_{h,0}$ is completely non-selfadjoint, which was shown in Theorem \ref{thm:sabc}. Hence, $\Imag \overline{g(0)}g'(0)\neq 0$, which implies in particular that $g(0)\neq 0$. Now, let $\eta:=\frac{-2i}{g(0)}g$ be the unique element in $\mbox{span}\{g\}$ which satisfies $\eta(0)=-2i$. Plugged into \eqref{eq:1}, this yields
		\begin{equation}
			\overline{\eta(0)}\eta'(0)=2i\eta'(0)=\int_0^\infty (-V_R(x)+iV_I(x)+\lambda)|\eta(x)|^2\mbox{d}x-\int_0^\infty|\eta'(x)|^2\mbox{d}x
		\end{equation}
		and by taking imaginary parts, we get
		\begin{equation} \label{eq:4}
			\Real \eta'(0)=\frac{1}{2}\int_0^\infty V_I(x)|\eta(x)|^2\mbox{d}x>0\:.
		\end{equation}
		We now claim that the operator $A_\lambda:=A_{\frac{i}{2}\eta'(0), V_I\eta}$ (defined in \eqref{eq:mdoext}) is a maximally dissipative extension of $A_{min}$, which has a non-trivial reducing selfadjoint subspace, which is given by $\ker(A_\lambda-\lambda)=\mbox{span}\{\eta\}$. Firstly, note that it follows from \eqref{eq:4} that we are in the critical case. It can now be directly verified that $\eta\in\cD(A_{\frac{i}{2}\eta'(0), V_I\eta})$ satisfies Conditions (i) and (ii) in Theorem \ref{thm:cnsa}. Thus, it has a reducing selfadjoint subspace spanned by $\eta$ corresponding to the eigenvalue $\lambda$. Uniqueness follows from the fact that there exists at most one solution to \eqref{eq:eivallambda} in $H^2(\R^+)$ (the problem is limit-circle at zero and limit-point at infinity).
	\end{proof}
	
	In what follows, we will work with the following definition of the essential spectrum of $A$ from \cite{EE18}.
	\begin{defn} The \textit{essential spectrum} of an operator $A$ is defined as
		\begin{equation*}
			\sigma_{ess}(A)=\{\lambda\in\C: \mbox{Ran}(A-\lambda)\:\mbox{is not closed or} \dim\ker(A-\lambda)=\infty \mbox{ or }\dim\ker(A^*-\overline{\lambda})=\infty\}\:.
		\end{equation*}
	\end{defn}
	\begin{remark}
		Note that we will only consider operators $A$ such that $\dim\ker(A-\lambda)$ and $\dim\ker(A^*-\overline{\lambda})$ are both at most one-dimensional. In this case, we get the following simpler description of the essential spectrum
		\begin{equation}
			\sigma_{ess}(A)=\{\lambda\in\C: \mbox{Ran}(A-\lambda)\:\mbox{is not closed}\}
		\end{equation}
		and in particular, since $\mbox{Ran}(A-\lambda)$ is closed if and only if $\mbox{Ran}(A^*-\overline{\lambda})$ is closed \cite[Chapter I, Thm.\ 5.13]{Kato}, we get $\sigma_{ess}(A)\cap\R=\sigma_{ess}(A^*)\cap\R$.
	\end{remark}
	\begin{prop}
		Let $\lambda\in\R$. If $\lambda\notin\sigma_{ess}(A_{min})$, then, the equation $A_{min}^*g=-g''+V_Rg-iV_Ig
		=\lambda g$ has a non-zero solution in $H^2(\R^+)$.
	\end{prop}
	\begin{proof}
		Assume that no such solution exists, i.e.\ $\ker(A_{min}^*-\lambda)=\{0\}$. This implies that $\overline{\mbox{Ran}(A_{min}-\lambda)}=\mathcal{H}$. Since $\lambda\notin\sigma_{ess}(A_{min})$, this implies $\mbox{Ran}(A_{min}-\lambda)=\mathcal{H}$. Since from the argument in the proof of Theorem \ref{prop:1} using \eqref{eq:1}, we have $\ker(A_{min}-\lambda)=\{0\}$, this would imply that $(A_{min}-\lambda)$ is boundedly invertible and thus $\lambda\in\rho(A_{min})$. Since resolvent sets are open, there exists an $\varepsilon>0$ such that $(\lambda-i\varepsilon)\in\rho(A_{min})$, thus implying that $A_{min}$ is maximally dissipative, which is a contradiction. Hence, $\mbox{Ran}(A_{min}-\lambda)$ is not closed and thus $\lambda\in\sigma_{ess}(A_{min})$.
	\end{proof}	
	Altogether, we have shown the following result:
	\begin{thm} \label{thm:final}
	 For any $\lambda\in\R\setminus\sigma_{ess}(A_{min})$, there exists a unique maximally dissipative extension $A_\lambda$ which has a one-dimensional reducing selfadjoint subspace corresponding to the eigenvalue $\lambda$.
	\end{thm}

	\section{Example} As an application, consider a finitely supported potential well of the form $V(x)=i\chi_{(0,1)}(x)$, i.e.
	\begin{equation}
		A_{min}:\quad\cD(A_{min})=\{f\in H^2(\R^+): f(0)=f'(0)=0\},\quad (A_{min}f)(x)=-f''(x)+i\chi_{(0,1)}(x)f(x)\:.
	\end{equation}
	Here, $\chi_{(0,1)}$ is the indicator function over the interval $(0,1)$.
	Firstly, let us argue that $\sigma_{ess}(A_{min})=[0,\infty)$. To see this, consider the extension $A_{\infty,0}$ of $A_{min}$, which is of the form $A_{\infty,0}=S_{\infty,0}+V$, where $S_{\infty,0}$ is the selfadjoint Dirichlet Laplacian on the half-line.  It is well-known that $\sigma_{ess}(S_{\infty,0})=[0,\infty)$. From a direct calculation, it can be seen that $V(S_{\infty,0}-i)^{-1}$ is Hilbert-Schmidt. This implies that $V$ is a relatively compact perturbation of $S_{\infty,0}$ and therefore $\sigma_{ess}(A_{\infty,0})=\sigma_{ess}(S_{\infty,0})=[0,\infty)$. Now, since $\dim(\cD(A_{\infty,0})/\cD(A_{min}))=1$, this implies that $\Ran(A_{min}-\lambda)$ is closed if and only if $\Ran(A_{\infty,0}-\lambda)$ is closed. Consequently, we have $\sigma_{ess}(A_{min})=[0,\infty)$ as well. According to Theorem \ref{thm:final}, for every $\lambda<0$, there is a unique maximally dissipative extension $A_\lambda$ of $A_{min}$ which has a non-trivial selfadjoint subspace corresponding to the eigenvalue $\lambda$. In what follows, we set $\lambda=-\xi^2$, where $\xi>0$.
	
	In order to determine $A_\lambda$, we need to find the function $\eta\in H^2(\R^+)$ with $\eta(0)=-2i$ such that
	\begin{equation}
		-\eta''(x)-i\chi_{(0,1)}(x)\eta(x)=-\xi^2\eta(x)\:.
	\end{equation} 
	We introduce the numbers $\sigma_\pm$, which are given by
	\begin{equation}
		\sigma_\pm=(\xi\pm\sqrt{\xi^2-i})\exp(\pm\sqrt{\xi^2-i})\:.
	\end{equation}
	Then, the solution to this initial value problem is given by
	\begin{equation}
		\eta(x)=\begin{cases}\frac{2i\sigma_-}{\sigma_+-\sigma_-}\exp(\sqrt{\xi^2-i}\cdot x)+\frac{-2i\sigma_+}{\sigma_+-\sigma_-}\exp(-\sqrt{\xi^2-i}\cdot x)\quad &\mbox{if } x\in(0,1)\\ \left[\frac{2i\sigma_-}{\sigma_+-\sigma_-}\exp(\sqrt{\xi^2-i}+\xi)+\frac{-2i\sigma_+}{\sigma_+-\sigma_-}\exp(-\sqrt{\xi^2-i}+\xi)\right]\exp(-\xi\cdot x)\quad &\mbox{if } x\in(1,\infty)\end{cases}\:.
	\end{equation}
	Let us now determine $h:=\frac{i}{2}\eta'(0)$ and $k(x):=\chi_{(0,1)}(x)\eta(x)$ in order to construct the maximally dissipative extension $A_\lambda=A_{h,k}$. We find
	\begin{align}
		h&=\frac{\sigma_-+\sigma_+}{\sigma_--\sigma_+}\sqrt{\xi^2-i}\qquad\mbox{    and }\notag\\ k(x)&=\begin{cases}\frac{2i\sigma_-}{\sigma_+-\sigma_-}\exp(\sqrt{\xi^2-i}\cdot x)+\frac{-2i\sigma_+}{\sigma_+-\sigma_-}\exp(-\sqrt{\xi^2-i}\cdot x)\quad &\mbox{if } x\in(0,1)\\ 0\quad &\mbox{if } x\in(1,\infty)\end{cases}\:,
	\end{align}
	which parametrise the special maximally dissipative extension $A_{h,k}$ of $A_{min}$ such that it has a non-trivial reducing selfadjoint subspace spanned by $\eta$ corresponding to the eigenvalue $\lambda=-\xi^2$.
	

\end{document}